\documentclass[preprint,12pt]{article}

\author{Jieliang Hong}
\title{Local behavior of local times of super Brownian motion}

\usepackage{graphicx}
\usepackage{amssymb}
\usepackage{amsmath}
\usepackage{amsthm}

\usepackage{float}
\usepackage{lineno}

\newtheorem{theorem}{Theorem}
\newtheorem{lemma}{Lemma}

\newtheorem{proposition}{Proposition}
\newtheorem{corollary}{Corollary}

\newcommand{\R}{\mathbb{R}}

\newcommand{\cF}{\mathcal F}

\begin{document}

\maketitle
\begin{abstract}
    For $x\in \R^d- \{0\}$, in dimension $d=3$, we study the asymptotic behavior of the local time $L_t^x$ of super-Brownian motion $X$ starting from $\delta_0$ as $x \to 0$. Let $\psi(x)=((1/2\pi^2) \log (1/|x|))^{1/2}$ be a normalization, Theorem 1 implies that $(L_t^x-(1/2\pi|x|))/\psi(x)$ converges in distribution to a standard normal distributed random variable as $x\to 0$. For dimension $d=2$, Theorem 2 implies that $L^x_t-(1/\pi)\log(1/|x|)$ is $L^1$ bounded as $x\to 0$. To do this, we prove a Tanaka formula for the local time which refines a result in Barlow, Evans and Perkins [1].
\end{abstract}

\section{Introduction and main results}

\subsection{Introduction}
Super Brownian Motion arises as a scaling limit of critical branching random walk. Let $M_F=M_F(\R^d)$ be the space of finite measures on $\R^d$ equipped with Borel $\sigma$- algebra $\mathfrak{B}(\R^d)$ and $(\Omega,\cF,\cF_t,P)$ be a filtered probability space. The $\mathit{Super}$-$\mathit{Brownian}$ $\mathit{ Motion}$ $X$ starting at $\mu\in M_F(\R^d)$ is a continuous $M_F(\R^d)$-valued adapted strong Markov process defined on $(\Omega,\cF,\cF_t,P)$ with $X_0=\mu$ a.s. which is the unique in law solution of a martingale problem (see (1) below).\\
For $0\leq t<\infty$, the weighted occupation time process is defined to be \[Y_t(A):=\int_0^t X_s(A) ds, \ A\in \mathfrak{B}(\R^d).\] If $\mu$ is a measure on $\R^d$ and $\psi$ is a real-valued function on $\R^d$, we write $\mu(\psi)$ for $\int_{\R^d} \psi(y) d\mu(y)$.\\
Local times of superprocesses have been studied by many authors. Sugitani [7] has proved that given the joint continuity of $\mu q_t(x)=\int \mu(dy) \int_0^t p_s(x-y) ds$ in $(t,s)$, the local time $L_t^x$ has a jointly continuous version which satisfies that for any $\phi \in C_b(\R^d)$, 
\[\int_0^t X_s(\phi) ds =\int_{\R^d}  L_t^x \phi(x) dx.\] $L_t^x$ is called the local time of $X$ at point $x\in \R^d$ and time $t>0$ and it also can be defined as \[L_t^x:=\lim_{\epsilon \to 0} \int_0^t X_s(p_\epsilon^x) ds,\] where $p_\epsilon^x(y)=p_\epsilon(y-x)$ is the transition density of Brownian motion. In general, for any fixed $\epsilon>0$, $L_t^x-L_\epsilon^x$ is jointly continuous in $t\geq \epsilon$ and $x\in \R^d$. \\
However, the condition of continuity of $\mu q_t(x)$ fails in $x=0$ when $\mu=\delta_0$ in $d=2$ and $d=3$ (joint continuity still holds for $L_t^x-L_\epsilon^x$). Our main result Theorem 1 gives precise information about the local behavior of local times of super-Brownian motion in dimension $d=3$. Let $x\in \R^d -\{0\}$ and $X$ be a super-Brownian motion initially in $\delta_0$, and $L_t^x$ be the local time of $X$ at time $t$ and point $x$. Theorem 1 tells us that as $x \to 0$ $L_t^x$ blows up like $1/|x|$  and has a variation like $\sqrt[]{\log 1/|x|}$. We can view this as an analogue to the classical Central Limit Theorem. For $d=2$, we derive a refined Tanaka formula in Proposition 3 compared to the one in [1] and Theorem 2 tells us that $\big|L_t^x-\frac{1}{\pi} \log 1/|x|\big|$ is $L^1$ bounded.

\subsection{Notations and Properties of super-Brownian motion}
We denote by $p_t(x)=(2\pi t)^{-d/2} e^{-|x|^2/2t}, t>0, x\in \R^d$ the transition density of d-dimensional Brownian motion $B_t$. Let $P_t$ be the corresponding Markov semigroup, then for any function $\phi$, \[P_t \phi(x)=\int p_t(y) \phi(x-y) dy.\]
Let $C_b^2(\R^d)$ denotes the set of all twice continuously differentiable functions on $\R^d$ with bounded derivatives of order less than 2. It is known that super-Brownian motion $X$ solves a martingale problem (Perkins [5], II.5): For any $ \phi \in C_b^2(\R^d)$, 
\begin{equation}
X_t(\phi)=X_0(\phi)+M_t(\phi)+\int_0^t X_s(\frac{\Delta}{2}\phi) ds,
\end{equation} where $M_t(\phi)$ is an $\cF_t$ martingale such that $M_0(\phi)=0$ and the quadratic variation of $M(\phi)$ is \[[M(\phi)]_t=\int_0^t X_s(\phi^2) ds.\]

For the first two moments of Super-Brownian motion, Konno and Shiga [4] gives us \[E_{X_0} X_t(\phi)=X_0(P_t\phi),\] and \[E_{X_0} \Big(X_t(\phi)^2\Big)=\Big(X_0(P_t\phi)\Big)^2+\int_0^t X_0\Big(P_s\big((P_{t-s}\phi)^2\big)\Big) ds.\]
We drop the subscript $X_0$ when there is no confusion.\\
\\
\noindent $\mathbf{Notations.}$ $c_3=1/2\pi$, $c_{3.1}=2c_3^2=1/2\pi^2$, $c_2=1/\pi$. The weird order here is to emphasize the dimension the constant is for.

\subsection{Main result}

\begin{theorem}(d=3)
Let $\psi(|x|)=(c_{3.1} \log 1/|x|)^{1/2}$, and $X$ be a super-Brownian motion in $\mathbb{R}^3$ with initial value $\delta_0$. Then for each $0<t\leq \infty$ as $x\to 0$, we have 

\begin{equation*}
\Big(X,\frac{L_t^x-c_3 \frac{1}{|x|}}{\psi(|x|)}\Big) \xrightarrow[]{d} \Big(X,Z\Big)
\end{equation*}
where $Z$ is a random variable with standard normal distribution and independent of $X$. Moreover, convergence in probability fails.
\end{theorem}

\begin{theorem}(d=2)
Let $X$ be a super-Brownian motion in $\mathbb{R}^2$ with initial value $\delta_0$. Then we have 
\[\limsup_{x \to 0} E\Big|L_t^x-c_2 \log\frac{1}{|x|}\Big | <\infty. \]
\end{theorem}

\section{Proof of Theorem 1}

Fix $x\in \R^3-{\{0\}}$, we will use the Tanaka formula for local times of super-Brownian motion (see [1], Theorem 6.1). Let $\phi_x(y)=c_3/|y-x|$, under the assumption $X_0(\phi_x)=\delta_0(\phi_x)=c_3/|x|<\infty,$ we have $P_{\delta_0}-$ almost surely that

\begin{equation}
L_t^x=c_3 \frac{1}{|x|}+M_t(\phi_x)-X_t(\phi_x),
\end{equation}
where $M_t(\phi_x)$ is an $\mathcal{F}_t$ martingale, with $M_0(\phi_x)=0$ and quadratic variation
\begin{equation}
[M(\phi_x)]_t=\int_0^t X_s(\phi_x^2) ds=\int_0^t \int \frac{c_3^2}{|y-x|^2} X_s(dy) ds.
\end{equation}
To prove Theorem 1, we need several propositions which are stated below and proofs of them will be shown in Section 2.2 after finishing the proof of Theorem 1.\\
\\
\noindent $\mathbf{Notations.}$ We define $g_x(y):=\log|y-x|$ for $x,y \in \R^3$.

\begin{proposition}
For $d=3$, we have almost surely that
\begin{equation}
X_t(g_x)=\delta_0(g_x)+ M_t(g_x)+\frac{1}{2} \int_0^t \int \frac{1}{|y-x|^2} X_s(dy) ds.
\end{equation}
\end{proposition}
\begin{proposition}
For $d>1$, we have
\begin{equation*}
\int_0^t \int \frac{1}{|y-x|} p_s(y) dy ds \leq \frac{2}{d-1} E|B_t|,  \ \ \  \forall x.
\end{equation*}
\end{proposition}

\subsection{Proof of Theorem 1}
Before proceeding to the proof, we state some lemmas which will be used in proving Theorem 1. 
\begin{lemma}
For any  $u,v \in \R^d-{\{0\}}$,  we have
\begin{equation*}
\Big|\log \frac{|u+v|}{|v|}\Big|  \leq \sqrt[]{\frac{|u|}{|v|}}+\sqrt[]{\frac{|u|}{|u+v|}}.
\end{equation*}
\end{lemma}
\begin{proof}
Let $f(u)=\sqrt[]{u}-\log(1+u)$ for $u\geq 0$. Observe that $f(0)=0$ and \[f'(u)=\frac{1}{2\sqrt[]{u}}-\frac{1}{1+u}=\frac{(\sqrt[]{u}-1)^2}{2\sqrt[]{u} (1+u)} \geq 0,\] therefore $f(u)\geq 0$ and $\log(1+u) \leq \sqrt[]{u}$ for all $u\geq 0$.\\
If $|u+v|\geq |v|$, then \[\Big|\log \frac{|u+v|}{|v|}\Big|=\log \frac{|u+v|}{|v|} \leq \log \frac{|u|+|v|}{|v|}\leq \sqrt[]{\frac{|u|}{|v|}}\leq \sqrt[]{\frac{|u|}{|v|}}+\sqrt[]{\frac{|u|}{|u+v|}}.\]
If $|u+v|\leq |v|$, then \[\Big|\log \frac{|u+v|}{|v|}\Big|=\log \frac{|v|}{|u+v|} \leq \log \frac{|v+u|+|u|}{|u+v|}\leq \sqrt[]{\frac{|u|}{|u+v|}}\leq \sqrt[]{\frac{|u|}{|v|}}+\sqrt[]{\frac{|u|}{|u+v|}}.\]
So Lemma 1 follows.

\end{proof}
\begin{lemma}
For any $t>0$, we have
\[\limsup_{x\to 0} E\Big[ \Big(\int \frac{1}{|y-x|} X_t(dy)\Big)^2  \Big]<\infty.\]

\end{lemma}
\begin{proof}
\begin{eqnarray*}
 E\Big[ \Big(\int \frac{1}{|y-x|} X_t(dy)\Big)^2  \Big] &=& \Big[\int  p_t(y) \frac{1}{|y-x|} dy \Big]^2\\
&+& \int_0^t \  ds  \int p_s(y) \ dy \Big( \int p_{t-s}(y-z) \frac{1}{|z-x|} \ dz \Big)^2 .
\end{eqnarray*}
For the first term,
\begin{eqnarray*}
& &\int  p_t(y) \frac{1}{|y-x|} dy\\
&\leq & 1+\int_{|y-x|<1} (\frac{1}{\sqrt[]{2\pi t}})^3 e^{\frac{-|y|^2}{2t}} \frac{1}{|y-x|}  dy\\
&\leq & 1+ (\frac{1}{\sqrt[]{2\pi t}})^3  \int_{\R^3}  \frac{1}{|y-x|} 1_{\{|y-x|<1\}} dy\\
&=& 1+ (\frac{1}{\sqrt[]{2\pi t}})^3 \  4\pi \int_0^1 r^2 \ dr  \ \frac{1}{r} <\infty. 
\end{eqnarray*}
For the second term, we use Cauchy Schwarz to get
\begin{eqnarray*}
& &(\int p_{t-s} (y-z) \frac{1}{|z-x|} dz)^2\\
&\leq & \int p_{t-s}(y-z) dz \cdot \int p_{t-s}(y-z) \frac{1}{|z-x|^2} dz\\
&=& \int p_{t-s}(y-z) \frac{1}{|z-x|^2} dz,
\end{eqnarray*}
and by Chapman-Kolmogorov
\begin{eqnarray*}
& & \int_0^t \  ds  \int p_s(y) \ dy \Big( \int p_{t-s}(y-z) \frac{1}{|z-x|} \ dz \Big)^2\\
&\leq & \int_0^t ds \int p_s(y) dy \int p_{t-s}(y-z) \frac{1}{|z-x|^2} dz\\
&=& \int_0^t ds \int \frac{1}{|z-x|^2} dz \int p_s(y) p_{t-s} (y-z) dy\\
&=& \int_0^t ds \int \frac{1}{|z-x|^2} \ dz \cdot p_t(z) =t \int \frac{1}{|z-x|^2}  p_t(z) dz.
\end{eqnarray*}
Using the same trick in the first term, we get\[\int \frac{1}{|z-x|^2}  p_t(z) dz \leq 1+ (\frac{1}{\sqrt[]{2\pi t}})^3 4\pi<\infty.\] 
Therefore we get \[\limsup_{x\to 0} E\Big[ \Big(\int \frac{1}{|y-x|} X_t(dy)\Big)^2  \Big]<\infty.\]
\end{proof}
\begin{lemma}
For any $t>0$,
\[\text{(i) } \limsup_{x\to 0} E\Big(X_t^2(g_{x})\Big) <\infty\] 
and 
\[ \text{(ii) }  \limsup_{x\to 0} E\Big(M_t^2(g_{x})\Big) <\infty.\]
\end{lemma}
\begin{proof}
{\ }
(i) For $|y-x|<1$, we bound $|g_{x}(y)|=\log 1/|y-x|$ by $1/|y-x|$, so
\begin{eqnarray*}
 & & \limsup_{x\to 0} E \Big[ \Big(\int_{|y-x|<1} \log |y-x| X_t(dy)\Big)^2\Big] \\
 & & \leq \limsup_{x\to 0} E\Big[\Big(\int \frac{1}{|y-x|} X_t(dy) \Big)^2\Big]<\infty
\end{eqnarray*}
 according to Lemma 2.\\
 
For $|y-x|\geq 1$, we bound $|g_{x}(y)|=\log |y-x|$ by $|y-x|$, so 
\begin{eqnarray*}
& & E \Big[ \Big(\int_{|y-x|\geq 1} \log |y-x| X_t(dy)\Big)^2\Big] \leq E\Big[\Big(\int |y-x| \  X_t(dy) \Big)^2\Big]\\
&=& \Big(\int p_t(y) |y-x| \ dy\Big)^2+\int_0^t ds \int p_s(z) dz \Big(\int |y-x| \  p_{t-s} (z-y) dy\Big)^2.
\end{eqnarray*}
It is clear that the first term is finite for any $x$ and for the second term,
\begin{eqnarray*}
& & \int_0^t ds \int p_s(z) dz \Big(\int p_{t-s} (z-y) |y-x| dy\Big)^2\\
&\leq &\int_0^t ds \int p_s(z) dz \int p_{t-s} (z-y) |y-x|^2 dy\\
&=& \int_0^t ds \int p_{t} (y) |y-x|^2 dy<\infty.
\end{eqnarray*}
So
\begin{eqnarray*}
& & \limsup_{x\to 0} E\Big[\Big(X_t(g_{x})\Big)^2\Big]\\
&=&\limsup_{x\to 0} E\Big[\Big(\int_{|y-x|<1} \log |y-x| \ X_t(dy) +\int_{|y-x|\geq 1} \log |y-x| \ X_t(dy)\Big)^2\Big]\\
&\leq &2 \  \limsup_{x\to 0}  E\Big[\Big(\int_{|y-x|<1} \log |y-x| \ X_t(dy)\Big)^2\Big] \\
 &   \ \ \         & +2 \  \limsup_{x\to 0}  E\Big[\Big(\int_{|y-x|\geq 1} \log |y-x| \ X_t(dy)\Big)^2\Big]<\infty.
\end{eqnarray*}
\\
(ii) Since $M_t(g_x)$ is a martingale with quadratic variation $[M(g_x)]_t=\int_0^t X_s(g_x^2) ds$, we get 
\begin{eqnarray*}
& & E\Big(M_t^2(g_x)\Big)=E \int_0^t X_s(g_x^2) \ ds= \int_0^t ds \int p_s(y) \Big(\log |y-x|\Big)^2 dy\\
&\leq & \int_0^t ds \int p_s(y) \frac{1}{|y-x|} 1_{\{|y-x|<1\}}dy+\int_0^t ds \int p_s(y) |y-x| 1_{\{|y-x|\geq 1\}}dy\\
&\leq & \int_0^t ds \int p_s(y) \frac{1}{|y-x|} dy+\int_0^t ds \int p_s(y) |y-x| dy. \ \ \ \ \ \  
(\star)
\end{eqnarray*}
We use the fact that $\log u\leq \log(1+u) \leq \sqrt{u}$ for $u \geq 1$ by Lemma 1.\\
By Proposition 2 in $d=3$, we get \[ \int_0^t ds \int p_s(y) \frac{1}{|y-x|} dy \leq E |B_t|<\infty.\]
As it is obvious that the latter term in $(\star)$ above is finite, we get \[ \limsup_{x\to 0} E\Big(M_t^2(g_{x})\Big) <\infty.\]
\end{proof}

\subsubsection{Convergence in distribution}
{\ }
Observe that combining (3) and (4), we obtain
\begin{equation*}
[M(\phi_x)]_t=2   c_3^2 \Big(X_t(g_x)-\delta_0(g_x)- M_t(g_x)\Big).
\end{equation*}
Note that $\delta_0(g_x)= \log |x|=-\log 1/|x|$, so
\begin{equation}
E\Big[\Big([M(\phi_x)]_t-c_{3.1} \log \frac{1}{|x|}\Big)^2\Big]= c_{3.1}^2 E\Big[\Big(X_t(g_x)-M_t(g_x)\Big)^2\Big]
\end{equation}
where $c_{3.1}=2  c_3^2$.
\begin{eqnarray*}
 & & E\Big[\Big(\frac{[M(\phi_x)]_t-c_{3.1} \log \frac{1}{|x|}}{c_{3.1} \log \frac{1}{|x|}}\Big)^2\Big]= \frac{c_{3.1}^2}{(c_{3.1} \log \frac{1}{|x|})^2} E\Big[\Big(X_t(g_x)-M_t(g_x)\Big)^2 \Big]\\
 &\leq &\frac{2}{(\log \frac{1}{|x|})^2} \Big[E\Big(X_t^2(g_x)\Big)+E\Big(M_t^2(g_x)\Big)\Big] \to  0 \text{ as } x \to 0,
\end{eqnarray*}
by Lemma 3. Hence we have shown that 
\begin{equation}
\frac{[M(\phi_x)]_t}{c_{3.1} \log \frac{1}{|x|}} \xrightarrow[]{L^2} 1 \text{ as } x\to 0.
\end{equation}
Since $\frac{[M(\phi_x)]_t}{c_{3.1} \log \frac{1}{|x|}}$ is the quadratic variation of martingale $\frac{M_t(\phi_x)}{\sqrt[]{c_{3.1} \log \frac{1}{|x|}}}$, using the Dubins-Schwarz theorem (see [6], Theorem V.1.6), we can find some  Brownian motion $B^{x} (t)$ in dimension 1 depending on $x$ such that
\[
\frac{M_t(\phi_x)}{\sqrt[]{c_{3.1} \log \frac{1}{|x|}}}=B^{x} \Big(\frac{[M(\phi_x)]_t}{c_{3.1} \log \frac{1}{|x|}}\Big).
\]
For any sequence ${\{x_n\}}$ that goes to 0, (6) implies that
\[\tau_n:=\frac{[M(\phi_{x_n})]_t}{c_{3.1} \log \frac{1}{|x_n|}} \xrightarrow[]{\text{ P}} 1 \text{ as } n \to \infty,\] and we claim that \[B_{\tau_n}^{x_n}=B^{x_n} \Big(\frac{[M(\phi_{x_n})]_t}{c_{3.1} \log \frac{1}{|x_n|}}\Big) \xrightarrow[]{d} Z,\] where $Z\sim N(0,1)$ in dimension 1.\\

In fact for any bounded uniformly continuous function $h(x)$, $\forall \ \epsilon>0, \exists \  \delta>0$ such that $|h(x)-h(y)|<\epsilon$ holds for any $x,y \in \R$ with $|x-y|<\delta$.  So
\begin{equation*}
E|h(B_{\tau_n}^{x_n})-h(B_1^{x_n})| \leq \epsilon+ 2 \|h\|_{\infty} \cdot P(|B_{\tau_n}^{x_n}-B_1^{x_n}|>\delta),\\
\end{equation*}
and for any $\gamma>0$, we have
\begin{eqnarray*}
 & & P(|B_{\tau_n}^{x_n}-B_1^{x_n}|>\delta) \\
 &\leq& P(|B_{\tau_n}^{x_n}-B_1^{x_n}|>\delta,|\tau_n-1|<\gamma)+P(|\tau_n-1|>\gamma)\\
 &\leq& P(\sup_{|s-1| \leq \gamma} |B_s^{x_n}-B_1^{x_n}|> \delta)+P(|\tau_n-1|>\gamma)\\
  &=& P(\sup_{|s-1| \leq \gamma} |B_s-B_1|> \delta)+P(|\tau_n-1|>\gamma)\\
  &<& \epsilon+P(|\tau_n-1|>\gamma), \text{ if we pick } \gamma \text{ small enough.}
\end{eqnarray*}
Since $\tau_n$ converge in probability to 1, for $n$ large enough,  we have $P(|\tau_n-1|>\gamma)<\epsilon$ and
so \[E|h(B_{\tau_n}^{x_n})-h(B_1^{x_n})|\leq \epsilon+ 2 \|h\|_{\infty} 2 \epsilon\] and hence
\begin{equation}
\frac{M_t(\phi_{x_n})}{\sqrt[]{c_{3.1} \log \frac{1}{|x_n|}}}=B_{\tau_n}^{x_n} \xrightarrow[]{d} Z,
\end{equation} where $Z\sim N(0,1)$. Recall that $\phi_{x_n}(y)=c_3 /|y-x_n|$ and by Lemma 2 \[\lim_{n\to \infty} E\Big[\Big(\frac{X_t(\phi_{x_n})}{c_{3.1} \log \frac{1}{|x_n|}}\Big)^2\Big]=0,\] hence
\begin{equation}
\frac{X_t(\phi_{x_n})}{\sqrt[]{c_{3.1} \log \frac{1}{|x_n|}}} \xrightarrow[]{\text{ p} } 0.
\end{equation}
Combining (7) and (8), by Theorem 25.4 in Billingsley [2] , we have 
\[\frac{L_t^{x_n}-\frac{1}{|x_n|}}{\sqrt[]{c_{3.1} \log \frac{1}{|x_n|}}}=
\frac{M_t(\phi_{x_n})}{\sqrt[]{c_{3.1} \log \frac{1}{|x_n|}}}- \frac{X_t(\phi_{x_n})}{\sqrt[]{c_{3.1} \log \frac{1}{|x_n|}}} \xrightarrow[]{d} Z.\]
So any sequence that approaches $0$ converges in distribution to $Z$ as above, which implies that \[\frac{L_t^{x}-\frac{1}{|x|}}{\sqrt[]{c_{3.1} \log \frac{1}{|x|}}} \xrightarrow[]{d} Z \text{ as } x\to 0.\]
For $t=\infty$, let $\rho$ be the life time of super Brownian motion $X$, then $L_{\infty}^x=L_\rho^x$. Chp II.5 in Perkins [5] tells us that $\rho<\infty$ a.s.. Sugitani [7] gives us \[L_t^x-L_\epsilon^x \text{ is continuous in } x \text{ for any } 0<\epsilon<t, \] with the initial condition being  $\delta_0$. \\

Fix $\epsilon$ small, we define $L_\rho^x-L_\epsilon^x=0$ if $\rho<\epsilon$. As $x\to 0$, we get 
 \[\frac{L_\rho^{x}-L_\epsilon^{x}}{(c_{3.1} \log \frac{1}{|x|})^{1/2}} \to 0\text{ a.s.},\] and by Theorem 25.4 in Billingsley [2] again we get\[\frac{L_\rho^{x}-\frac{c_3}{|x|}}{(c_{3.1} \log \frac{1}{|x|})^{1/2}} =\frac{L_\rho^{x}-L_\epsilon^{x}}{(c_{3.1} \log \frac{1}{|x|})^{1/2}} +\frac{L_\epsilon^{x}-\frac{c_3}{|x|}}{(c_{3.1} \log \frac{1}{|x|})^{1/2}} \xrightarrow[]{d} Z.\]
$\hfill\square$

\subsubsection{Remaining Part of Theorem 1}

(i) Fix $0<t\leq \infty$, let $Z_t^{x_n}$  denotes $(L_t^{x_n}-c_3/|x_n|)/(c_{3.1}\log 1/|x_n|)^{1/2}$. By tightness of each component in $(X,Z^{x_n}_t)$, we clearly have tightness of $(X,Z^{x_n}_t)$ as $x_n\to 0$, so it suffices to show all weak limit points coincide. Assume $(X,Z^{x_n}_t)$ converges weakly to $(X,Z)$ for some sequence $x_n\to 0$. Let $(X,Z)$ be defined on $(\tilde{\Omega},\tilde{\cF}_t,\tilde{P})$ where $X$ is super-Brownian motion and $Z$ is standard normal under $\tilde{P}$. \\

For any $0<t_1<t_2<\cdots<t_m$, let $\phi_0: \R\to \R$ and $\psi_i: M_F \to \R$, $1\leq i\leq m$ be bounded continuous, we have \[\lim_{n\to\infty} E\Big[\psi_1(X_{t_1})\cdots \psi_m(X_{t_m}) \phi_0(Z^{x_n}_t)\Big]=\tilde{E}\Big[\psi_1(X_{t_1})\cdots \psi_m(X_{t_m}) \phi_0(Z)\Big]\] since we assume that $(X,Z^{x_n}_t)$ converge weakly to $(X,Z)$.\\
 Pick $\epsilon>0$ such that $\epsilon<t_1$ and $\epsilon<t$, by Sugitani [7], \[L_t^x-L_\epsilon^x \text{ is continuous in } x \text{ for any } 0<\epsilon<t \] with the initial condition being  $\delta_0$,  when $n\to \infty$ we get 
 \[Z_t^{x_n}-Z_\epsilon^{x_n}=\frac{L_t^{x_n}-L_\epsilon^{x_n}}{(c_{3.1} \log \frac{1}{|x_n|})^{1/2}} \to 0\text{ a.s.}. \] and hence \[(0,Z_t^{x_n}-Z_\epsilon^{x_n}) \to (0,0) \text{ a.s.. }\]  
By Theorem 25.4 in Billingsley [2] again  \[(X,Z_\epsilon^{x_n})=(X,Z_t^{x_n})-(0,Z_t^{x_n}-Z_\epsilon^{x_n}) \text{ converge weakly to } (X,Z). \]
 Therefore  since $Z_\epsilon^{x_n} \in \cF_\epsilon^{X}$,
\begin{eqnarray*}
 I &=& \tilde{E} \Big[\psi_1(X_{t_1})\cdots \psi_m(X_{t_m})  \cdot \phi_0(Z)\Big]\\
 &=& \lim_{n\to \infty} E \Big[\psi_1(X_{t_1})\cdots \psi_m(X_{t_m})  \cdot \phi_0(Z_\epsilon^{x_n})\Big]\\
  &=& \lim_{n\to \infty} E \Big[E \Big( \psi_1(X_{t_1})\cdots \psi_m(X_{t_m}) \big| \cF_\epsilon^X \Big) \cdot \phi_0(Z_\epsilon^{x_n})\Big]\\
  &=&\lim_{n\to \infty} E \Big[E_{X_\epsilon} \Big( \prod_{i=1}^m \psi_i(X_{t_i-\epsilon})\Big) \cdot \phi_0(Z_\epsilon^{x_n})\Big]
\end{eqnarray*}
Define \[F_\epsilon(\mu)=E_\mu\Big(\prod_{i=1}^m \psi_i(X_{t_i-\epsilon})\Big)\] for $\mu\in M_F$ and we prove by induction that $F_\epsilon \in C_b(M_F)$. For $m=1$ we have \[F_\epsilon(\mu)=E_\mu \Big( \psi_1(X_{t_1-\epsilon})\Big)=P_{t_1-\epsilon} \psi_1(\mu). \]
By Theorem II.5.1 in Perkins [5], if $P_t F(\mu) = E_\mu F(X_t)$, then $P_t: C_b(M_F) \to C_b(M_F)$ so $F_\epsilon = P_{t_1-\epsilon} \psi_1 \in C_b(M_F)$ since $\psi_1 \in C_b(M_F)$. Suppose it holds for $m-1$, then
\begin{eqnarray*}
F_\epsilon(\mu)&=&E_\mu \Big( \prod_{i=1}^m \psi_i(X_{t_i-\epsilon}) \Big)\\
&=& E_\mu \Big[ \prod_{i=1}^{m-2} \psi_i(X_{t_i-\epsilon}) \cdot E_\mu \Big( \psi_{m-1}(X_{t_{m-1}-\epsilon}) \psi_m(X_{t_m-\epsilon}) \big| \cF_{t_{m-1}-\epsilon}^X \Big)\Big]\\
&=& E_\mu \Big[ \prod_{i=1}^{m-2} \psi_i(X_{t_i-\epsilon}) \cdot \psi_{m-1}(X_{t_{m-1}-\epsilon}) P_{t_m-t_{m-1}} \psi_m(X_{t_{m-1}-\epsilon})  \Big]\\
&=& E_\mu \Big[ \prod_{i=1}^{m-2} \psi_i(X_{t_i-\epsilon})  \cdot \tilde{\psi}_{m-1}(X_{t_{m-1}-\epsilon})  \Big]
\end{eqnarray*}
where $\tilde{\psi}_{m-1} $ defined to be $\psi_{m-1} P_{t_m-t_{m-1}} \psi_m$ is in $C_b(M_F)$. It is reduced to the case $m-1$ where we already have $F_\epsilon \in C_b(M_F)$, so it holds for case $m$.\\ 

Therefore by the weak convergence of $(X,Z_\epsilon^{x_n})$ to $(X,Z)$, we have \[ \lim_{n\to \infty} E \Big[F_\epsilon(X_\epsilon) \cdot \phi_0(Z_\epsilon^{x_n})\Big]=\tilde{E} \Big[F_\epsilon(X_\epsilon) \cdot \phi_0(Z)\Big]\] and hence
\begin{eqnarray*}
I &=& \lim_{n\to \infty} E \Big[E_{X_\epsilon} \Big( \prod_{i=1}^m \psi_1(X_{t_i-\epsilon})\Big) \cdot \phi_0(Z_\epsilon^{x_n})\Big]\\
&=& \lim_{n\to \infty} E \Big[F_\epsilon(X_\epsilon) \cdot \phi_0(Z_\epsilon^{x_n})\Big]=\tilde{E} \Big[F_\epsilon(X_\epsilon) \cdot \phi_0(Z)\Big]\\
&=&\tilde{E} \Big[\tilde{E}_{X_\epsilon} \Big( \prod_{i=1}^m \psi_1(X_{t_i-\epsilon})\Big) \cdot \phi_0(Z)\Big]\\
&=& \tilde{E} \Big[\tilde{E}  \Big( \psi_1(X_{t_1})\cdots \psi_m(X_{t_m})  \big| \tilde{\cF}_\epsilon^X \Big)  \cdot \phi_0(Z)\Big]
\end{eqnarray*}
Let $\epsilon \to 0$, by martingale convergence we have \[\tilde{E}  \Big( \psi_1(X_{t_1})\cdots \psi_m(X_{t_m})  \big| \tilde{\cF}_\epsilon^X \Big) \xrightarrow[]{L^1 } \tilde{E}  \Big( \psi_1(X_{t_1})\cdots \psi_m(X_{t_m}) \big| \tilde{\cF}_{0+}^X\Big)=\tilde{E}  \Big( \psi_1(X_{t_1})\cdots \psi_m(X_{t_m}) \Big).\] The equality follows from Blumental 0-1 law that $\tilde{\cF}_{0+}^X$ is trivial. Therefore
\begin{align*}
 I=& \tilde{E} \Big[\psi_1(X_{t_1})\cdots \psi_m(X_{t_m})  \cdot \phi_0(Z)\Big]\\
 =& \lim_{\epsilon \to 0 }\tilde{E} \Big[ \tilde{E}  \Big( \psi_1(X_{t_1})\cdots \psi_m(X_{t_m})  \big| \tilde{\cF}_\epsilon^X \Big) \cdot \phi_0(Z)\Big]\\
 =& \tilde{E} \Big[\tilde{E}  \Big( \psi_1(X_{t_1})\cdots \psi_m(X_{t_m}) \Big) \cdot  \phi_0(Z)\Big]\\ 
 =& \tilde{E} \Big( \psi_1(X_{t_1})\cdots \psi_m(X_{t_m}) \Big) \cdot \tilde{E} \phi_0(Z)
\end{align*}

The above functionals are a determining class on $C([0,\infty),M_F) \times R$ and so we get weak
convergence of $(X,Z^x_t) \to (X,Z)$ where the latter are independent.\\

(ii) Suppose we find convergence in probability for $0<t\leq \infty$, \[\frac{L_t^{x_n}-\frac{1}{|x_n|}}{(c_{3.1} \log \frac{1}{|x_n|})^{1/2}} \xrightarrow[]{\text{ P} } Z\] for some random variable $Z$, then it must converge in distribution to $Z$ as well, so $Z$ is a standard normal distributed random variable. By taking a further subsequence we may assume a.s. convergence holds:\[\frac{L_t^{x_n}-\frac{1}{|x_n|}}{(c_{3.1} \log \frac{1}{|x_n|})^{1/2}} \xrightarrow[]{\text{ a.s. } } Z.\]
By Sugitani [7], \[L_t^x-L_\epsilon^x \text{ is continuous in } x \text{ for any } 0<\epsilon<t \] with the initial condition being  $\delta_0$,  we get 
 \[\frac{L_t^{x_n}-L_\epsilon^{x_n}}{(c_{3.1} \log \frac{1}{|x_n|})^{1/2}} \to 0\text{ a.s.}. \] 
 Therefore
\begin{equation}
\frac{L_\epsilon^{x_n}-\frac{1}{|x_n|}}{(c_{3.1} \log \frac{1}{|x_n|})^{1/2}} =\frac{L_t^{x_n}-\frac{1}{|x_n|}}{(c_{3.1} \log \frac{1}{|x_n|})^{1/2}} -\frac{L_t^{x_n}-L_\epsilon^{x_n}}{(c_{3.1} \log \frac{1}{|x_n|})^{1/2}} \to Z \text{ a.s.}
\end{equation}
Because (9) holds for any $\epsilon>0$, we get \[Z \in \bigcap_{t>0} \cF_t^X=\cF_{0+}^X,\] and Blumenthal 0-1 law tells us that any event in $\cF_{0+}^X$ is an event of probability $0$ or $1$, hence $Z$ is a.s. constant. This contradicts the fact that $Z$ is standard normal. So we get a contradiction by assuming that $(L_t^{x_n}-c_3/|x_n|)/(c_{3.1} \log \frac{1}{|x_n|})^{1/2}$ converges in probability. $\hfill\square$



\subsection{Proof of Proposition 1 and 2}

\subsubsection{Some useful lemmas}

\begin{lemma}
For any $0<\alpha<3$, there exists a constant $C=C(\alpha)$ such that for any $x\neq 0$ and $t>0$, 
\[
\int_{\R^3} p_t(y) \frac{1}{|y-x|^\alpha} dy<C \frac{1}{|x|^\alpha}.
\]
\end{lemma}

\begin{proof}
Fix $\delta=|x|/2$,
\begin{eqnarray*}
 & & \int_{\R^3} p_t(y) \frac{1}{|y-x|^\alpha} dy \\
 &\leq& \frac{1}{\delta^\alpha}+  \int_{|y-x|<\delta} p_t(y) \frac{1}{|y-x|^\alpha} dy.
\end{eqnarray*}
For $|y-x|<\delta$, we have $|y|\geq |x|-|y-x|>|x|-\delta=\delta$, therefore
\begin{eqnarray*}
& & \int_{|y-x|<\delta} (\frac{1}{2\pi t})^{3/2}  e^{-\frac{|y|^2}{2t}} \frac{1}{|y-x|^\alpha} dy\\
 &\leq& \int_{|y-x|<\delta} (\frac{1}{2\pi t})^{3/2}  e^{-\frac{\delta^2}{2t}} \frac{1}{|y-x|^\alpha} dy\\
 &=& (\frac{1}{2\pi t})^{3/2} e^{-\frac{\delta^2}{2t}} \int_0^\delta \frac{1}{r^\alpha} r^2  dr \cdot 4\pi\\
 &=& 4\pi  M(\delta) \cdot \frac{1}{3-\alpha} \delta^{3-\alpha}
\end{eqnarray*}
where 
\[M(\delta):=\sup_{t> 0} (\frac{1}{2\pi t})^{3/2} e^{-\frac{\delta^2}{2t}}
\overset{\underset{  u=\frac{\delta^2}{t}  }{}}{=}\frac{1}{\delta^3} \sup_{u\geq 0} (\frac{u}{2\pi})^{3/2} e^{-u/2}:=C_0 \frac{1}{\delta^3}.\]
Therefore 
\begin{equation*}
\int_{\R^3} p_t(y) \frac{1}{|y-x|^\alpha} dy < \frac{1}{\delta^\alpha}+4  \pi \cdot C_0 \frac{1}{\delta^3} \cdot \frac{1}{3-\alpha} \delta^{3-\alpha} =C(\alpha) \frac{1}{|x|^\alpha}.
\end{equation*}
\end{proof}

\begin{corollary}
For any $0<\alpha<3$, there exists a constant $C=C(\alpha)$ such that for any $x\neq 0$ and $t>0$,
\[
E\int_0^t \int \frac{1}{|y-x|^\alpha} X_s(dy) ds =\int_0^t ds \int_{\R^3} p_s(y) \frac{1}{|y-x|^\alpha} dy<C \frac{1}{|x|^\alpha} t.
\]
\end{corollary}

\begin{proof}
It directly follows from Lemma 4.
\end{proof}

\begin{lemma}
In $\R^3$, for any fixed $s>0$ and $y\neq x$, we have \[\Delta_y P_s g_x(y)= \int p_s(y-z) \frac{1}{|z-x|^2}dz.\]
\end{lemma}
\begin{proof}
Idea of this proof is from Evans [3]. For any fixed $s>0$, $p_{s}(y)=(2\pi s)^{-3/2} e^{-|y|^2/2s} \in C_0^\infty(\R^3)$, we have \[\|Dp_s\|_{L^\infty(\R^3)}<\infty \text{ and }\|\Delta p_s\|_{L^\infty(\R^3)}<\infty.\] Here $Du=D_x u=(u_{x_1},u_{x_2}, u_{x_3})$ denotes the gradient of $u$ with respect to $x=(x_1,x_2,x_3).$
\\

For any $\delta \in (0,1)$, 
\begin{eqnarray*}
& & \Delta_y \int_{\R^3} p_{s}(y-z) g_x(z)dz\\
&=&  \int_{B(x,\delta)} \Delta_y p_{s}(y-z) g_x(z)dz+\int_{\R^3-B(x,\delta)} \Delta_y p_{s}(y-z) g_x(z)dz\\
&=:& I_\delta+J_\delta.
\end{eqnarray*}
\\
Now \[|I_\delta|\leq \|\Delta p_s\|_{L^\infty(\R^3)} \int_{B(x,\delta)} |g_x(z)| dz\leq C \delta^3 |\log \delta| \to 0.\]
Note that $\Delta_y p_{s}(y-z)=\Delta_z p_{s}(y-z)$. Integration by parts yields
\begin{eqnarray*}
J_\delta &=& \int_{\R^3-B(x,\delta)} \Delta_z p_{s}(y-z) g_x(z)dz\\
&=& \int_{\partial B(x,\delta)} g_x(z) \frac{\partial p_{s}}{\partial \nu}(y-z) dz-\int_{\R^3-B(x,\delta)} D_z p_{s}(y-z) D_z g_x(z)dz\\
&=:& K_\delta+L_\delta,
\end{eqnarray*}
$\nu$ denoting the inward pointing unit normal along $\partial B(x,\delta).$ So
\[|K_\delta|\leq \|Dp_s\|_{L^\infty(\R^3)} \int_{\partial B(x,\delta)} |g_x(z)|dz \leq C \delta^2 |\log \delta| \to 0.   \]
We continue by integrating by parts again in the term $L_\delta$ to find
\begin{eqnarray*}
L_\delta &=& \int_{\R^3-B(x,\delta)} p_{s}(y-z) \Delta_z g_x(z)dz-\int_{\partial B(x,\delta)} p_{s}(y-z) \frac{\partial g_x}{\partial \nu}(z) dz\\
&=:& M_\delta+N_\delta.
\end{eqnarray*}
Now $Dg_x(z)=\frac{z-x}{|z-x|^2}(z\neq x)$ and $\nu=\frac{-(z-x)}{|z-x|}=\frac{-(z-x)}{\delta}$ on $\partial B(x,\delta)$. Hence $\frac{\partial g_x}{\partial \nu}(z)=\nu \cdot Dg_x(z)=-\frac{1}{\delta} $ on $\partial B(x,\delta)$. Since $4\pi \delta^2$ is the surface area of the sphere $\partial B(x,\delta)$ in $\R^3$, we have
\[N_\delta=4\pi \delta \cdot \frac{1}{4\pi \delta^2} \int_{\partial B(x,\delta)} p_{s} (y-z)dz \to 0\cdot p_{s}(y-x) = 0 \text{ as } \delta \to 0.\]
By direct calculation, we have $\Delta_z g_x(z)=\frac{1}{|x-z|^2}$ when $z \in \R^3-B(x,\delta)$, therefore
\[M_\delta=\int_{\R^3-B(x,\delta)} p_{s}(y-z) \frac{1}{|x-z|^2}  dz.\] 
Lemma 4 gives  \[\int p_{s}(y-z) \frac{1}{|x-z|^2}  dz<\infty,\] by Dominated Convergence Theorem, we have \[M_\delta=\int p_{s}(y-z) \frac{1}{|x-z|^2} 1_{\{|z-x|\geq \delta\}} dz \to \int p_{s}(y-z) \frac{1}{|x-z|^2}  dz\]
as $\delta \to 0$.
\end{proof}

\subsubsection{Proof of Proposition 1}
Define $\eta\in C^\infty(\R^d)$ by \[\eta(x):=C \exp\Big(\frac{1}{|x|^2-1}\Big) 
1_{\{|x|<1\} },\] the constant $C$ selected such that $\int_{\R^d} \eta dx=1$.\\
Let $\chi_{n}$ be the convolution of $\eta$ and the indicator function of the ball $B_{n}=\{x: |x|<n\}$, we get
\[\chi_{n} (x)=\int_{\R^d} 1_{\{|x-y|<n\}} \eta (y) dy=\int_{B_{1}} 1_{\{|x-y|<n\}} \eta(y) dy.\]
It is known that $\chi_{n}$ is a $C^{\infty}$ function with support in $B_{n+1}$ and for $x\in B_{n-1}$, we have $|x-y|<n$ since $|x|<n-1$ and $|y|<1$, so \[\chi_{n} (x)=\int_{B_{1}} 1_{\{|x-y|<n\}} \eta(y) dy=\int_{B_{1}} \eta(y) dy=1.\] It's easy to see that $\chi_n$ increases to $1$ as $n$ goes to infinity.

Recall that $g_x(y)=\log|y-x|$ and let $g_{n,x}(y)=g_x(y) \cdot \chi_n(y-x)$, then
\begin{align*}
    P_{\epsilon} g_{n,x}(z)=&\int_{|y-x|<n-1} p_\epsilon(z-y) \log |y-x| dy\\
   &+ \int_{n-1<|y-x|<n+1} p_\epsilon(z-y) \log |y-x| \chi_n(y-x) dy \in C_b^{2},
\end{align*}
and
\begin{align*}
    \Delta_z P_{\epsilon} g_{n,x}(z)=&\int_{|y-x|<n-1}  \Delta_z p_\epsilon(z-y) \log |y-x| dy\\
   &+ \int_{n-1<|y-x|<n+1}  \Delta_z p_\epsilon(z-y) \log |y-x| \chi_n(y-x) dy.
\end{align*} 
It is easy to see that $P_{\epsilon} g_{n,x}(z)$ and $\Delta_z P_{\epsilon} g_{n,x}(z)$ increases to $P_{\epsilon} g_{x}(z)$ and $\Delta_z P_{\epsilon} g_{x}(z)$ respectively.\\

For $P_{\epsilon} g_{n,x} \in C_b^2(\R^3)$, we have following equation hold a.s., 
 \begin{equation*}
 X_t(P_{\epsilon} g_{n,x} )=\delta_0(P_{\epsilon} g_{n,x})+ M_t(P_{\epsilon} g_{n,x})+\int_0^t X_s(\frac{\Delta}{2} P_{\epsilon} g_{n,x})  ds,
 \end{equation*}
 where $M_t(P_{\epsilon} g_{n,x})$ is a martingale with quadratic variation \[[M(P_{\epsilon} g_{n,x})]_t=\int_0^t X_s\Big((P_{\epsilon} g_{n,x})^2\Big) ds.\]
 As $n$ goes to infinity, by monotone convergence, we have \[X_t(P_{\epsilon} g_{n,x} ) \to X_t(P_{\epsilon} g_{x} ),\ \delta_0(P_{\epsilon} g_{n,x})\to \delta_0(P_{\epsilon} g_{n,x}),\]and \[\int_0^t X_s(\frac{\Delta}{2} P_{\epsilon} g_{n,x})  ds \to \int_0^t X_s(\frac{\Delta}{2} P_{\epsilon} g_{x})  ds.\]
Note that 
 \begin{eqnarray*}
  & & E \int_0^t X_s\Big((P_{\epsilon} g_x)^2\Big) ds\\
  &=& \int_0^t ds \int p_s(y) dy \Big(\int p_\epsilon(y-z) \log |z-x| dz \Big)^2\\
  &\leq&  \int_0^t ds \int p_s(y) dy \int p_\epsilon(y-z) \Big(\log |z-x| \Big)^2 dz \\
  &=&  \int_0^t ds \int p_{s+\epsilon}(z)  \Big(\log |z-x| \Big)^2 dz \\
   &\leq&  \int_0^{t+\epsilon} ds \int p_{s}(z)  \Big(\log |z-x| \Big)^2 dz <\infty.
 \end{eqnarray*}
The last is by $(\star)$ in Lemma 3 when calculating $E(M_t^2(g_x))$. So we conclude that \[E\Big[\Big(M_t(P_{\epsilon} g_{n,x})-M_t(P_{\epsilon} g_x)\Big)^2\Big]=E\int_0^t X_s\Big((P_{\epsilon}g_{n,x}-P_{\epsilon}g_x)^2\Big)ds \to 0\] by Dominated Convergence Theorem since \[(P_{\epsilon}g_{n,x}-P_{\epsilon}g_x)^2 \to 0 \text{ and } (P_{\epsilon}g_{n,x}-P_{\epsilon}g_x)^2 \leq 4 (P_{\epsilon} g_x)^2.\]
So the $L^2$ convergence of a martingale $M_t(P_{\epsilon} g_{n,x})$to $M_t(P_{\epsilon} g_{x})$ follows, which makes $M_t(P_{\epsilon} g_{x})$ a martingale as well. By taking a subsequence we have the following equation holds a.s.
  \begin{equation}
 X_t(P_{\epsilon} g_x)=\delta_0(P_{\epsilon} g_x)+ M_t(P_{\epsilon} g_x)+\int_0^t X_s(\frac{\Delta}{2} P_{\epsilon} g_x) ds,
 \end{equation}
 where $M_t(P_{\epsilon} g_x)$ is a martingale with integrable quadratic variation \[[M(P_{\epsilon} g_x)]_t=\int_0^t X_s\Big((P_{\epsilon} g_x)^2\Big) ds.\]

Let $\epsilon$ goes to $0$, we will show in (i)-(iv) the $L^1$ convergence of each term in (10) to the corresponding term in Proposition 1, i.e. \[X_t(g_x)=\delta_0(g_x)+ M_t(g_x)+\frac{1}{2} \int_0^t \int \frac{1}{|y-x|^2} X_s(dy) ds.\]
(i)

First we have 
\begin{eqnarray*}
& & \bigg|\delta_0(P_\epsilon g_x)-\delta_0(g_x)\bigg|=\bigg|\int_{\R^2}  p_\epsilon(y) \log|y-x| dy-\log|x| \bigg|\\
&\leq & \int_{\R^3}  p_\epsilon(y) \Big|\log|y-x| -\log|x|\Big| dy=E\Big[\Big|\log|B_\epsilon-x|-\log|x|\Big|\Big]\\
&=& E\Big[\Big|\log\frac{|B_\epsilon-x|}{|x|}\Big|\Big].
\end{eqnarray*}

As a result,
\begin{eqnarray*}
 & & \bigg|\delta_0(P_\epsilon g_x)-\delta_0(g_x)\bigg| \leq E\Big[\Big|\log\frac{|B_\epsilon-x|}{|x|}\Big|\Big]\\ 
 &\leq& E \Big[\sqrt[]{\frac{|B_\epsilon|}{|x|}}\Big]+E\Big[\sqrt[]{\frac{|B_\epsilon|}{|B_\epsilon-x|}}\Big] \text{ by Lemma 1}\\
 &\leq& \frac{1}{|x|^{\frac{1}{2}}} E|B_\epsilon|^{\frac{1}{2}} + \Big(E|B_\epsilon|\Big)^{\frac{1}{2}}\cdot \Big(E\frac{1}{|B_\epsilon-x|}\Big)^{\frac{1}{2}} \\
 &\leq & \frac{1}{|x|^{\frac{1}{2}}} E|B_\epsilon|^{\frac{1}{2}} + \Big(E|B_\epsilon|\Big)^{\frac{1}{2}}\cdot \Big(C \frac{1}{|x|}\Big)^{\frac{1}{2}} \text{ by Lemma 4}\\
 &\to& 0 \text{ as } \epsilon \to 0.
\end{eqnarray*}
(ii)

Let $B_t$ and $B_t^{'}$ be two independent standard Brownian motion in $\R^3$, 
\begin{eqnarray*}
& & E\Big[\Big|X_t(P_\epsilon g_x)-X_t(g_x)\Big|\Big] \leq E \Big[X_t\Big(|P_\epsilon g_x-g_x|\Big)\Big]\\
&=&  \int p_t(y) dy \bigg|\int p_\epsilon(z) \log|z-(y-x)|dz-\log|y-x|\bigg|\\
&\leq&  \int p_t(y) dy \int p_\epsilon(z) \bigg|\log|z-(y-x)|-\log|y-x|\bigg| dz\\
&=&  E\Bigg(\Big|\log \frac{|B'_\epsilon-(B_t-x)|}{|B_t-x|}\Big|\Bigg)\\
&\leq&   E \Big[\sqrt{\frac{|B'_\epsilon| }{|B_t-x|}}\Big]+ E \Big[\sqrt{\frac{|B'_\epsilon|}{|B'_\epsilon+B_t-x|}}\Big] .
\end{eqnarray*}
Since $E\sqrt[]{|B'_\epsilon|} \to 0$ and by Lemma 4
\[E \sqrt{\frac{|B'_\epsilon| }{|B_t-x|}}= E \sqrt[]{|B'_\epsilon|} \cdot E\sqrt[]{\frac{1}{|B_t-x|}} \leq E \sqrt[]{|B'_\epsilon|} \cdot C\frac{1}{|x|^{\frac{1}{2}}}\to 0.\]
For the second term, we use Cauchy Schwarz Inequality,
\[\bigg(E \sqrt{\frac{|B'_\epsilon|}{|B'_\epsilon+B_t-x|}}\bigg)^2 \leq E|B'_\epsilon| \cdot E\frac{1}{|B'_\epsilon+B_t-x|}= E|B'_\epsilon| \cdot E\frac{1}{|B_{t+\epsilon}-x|}.\]
So again by Lemma 4\[E \sqrt{\frac{|B'_\epsilon|}{|B'_\epsilon+B_t-x|}} \leq \bigg(E|B'_\epsilon|\bigg)^{1/2} \cdot \bigg(C \frac{1}{|x|^{\frac{1}{2}}}\bigg)^{1/2} \to 0 \text{ as } \epsilon \to 0.\]
and the $L^1$ convergence of $X_t(P_\epsilon g_x)$ to $X_t(g_x)$ follows.\\
\\
(iii)

Next we deal with $M_t(P_\epsilon g_x)-M_t(g_x)$ and we use its quadratic variation to compute its second moment.
\begin{eqnarray*}
& & \Big(E|M_t(P_\epsilon g_x)-M_t(g_x)|\Big)^2 \leq E \Big[\Big(M_t(P_\epsilon g_x)-M_t(g_x)\Big)^2\Big]\\
&=& E \int_0^t X_s\Big((P_\epsilon g_x-g_x)^2\Big) ds\\
&=& \int_0^t ds \int p_s(y) dy \bigg( \int p_\epsilon(z)  \Big(\log|z+y-x| -\log|y-x|\Big) dz\bigg)^2 \\
&\leq& \int_0^t ds \int p_s(y) dy \int p_\epsilon(z)  \Big(\log|
z+y-x| -\log|y-x|\Big)^2 dz \\
&=& \int_0^t E \Big[\Big(\log |B_\epsilon^{'}+B_s-x|-\log |B_s-x|\Big)^2\Big] ds.
\end{eqnarray*}
By Lemma 1 we get
\begin{eqnarray*}
 & &  \int_0^t E \Big[\Big(\log \frac{|B_\epsilon^{'}+B_s-x|}{|B_s-x|}\Big)^2\Big] ds\\
 &\leq&  \int_0^t  E \Big[\Big(\sqrt{\frac{|B_\epsilon^{'}|}{|B_s-x|}}+\sqrt{\frac{|B_\epsilon^{'}|}{|B'_\epsilon+B_s-x|}}\Big)^2\Big] ds \\
 &\leq& 2 \int_0^t  E \Big({\frac{|B_\epsilon^{'}|}{|B_s-x|}}\Big)+E\Big({\frac{|B_\epsilon^{'}|}{|B'_\epsilon+B_s-x|}}\Big)ds\\
 &:=& 2I.
\end{eqnarray*}
For the first term in $I$,
\begin{eqnarray*}
 & &  \int_0^t E {\frac{|B_\epsilon^{'}|}{|B_s-x|}}ds  =E|B_\epsilon^{'}|
 \cdot \int_0^t E {\frac{1}{|B_s-x|}}ds\\
 &\leq & C \epsilon^{1/2} \int_0^t ds \int p_s(y) \frac{1}{|y-x|} dy \to 0  \text{ as } \epsilon \to 0 \ \text{ by Corollary 1} .
\end{eqnarray*}
For the second term in $I$, note that $B'_\epsilon+B_s=^d B_{s+\epsilon}$ as they are independent Brownian motion, so
\begin{eqnarray*}
 &  &  \int_0^t  E {\frac{|B_\epsilon^{'}|}{|B'_\epsilon+B_s-x|}}ds \leq \int_0^t  \Big(E\big(|B'_\epsilon|^2\big)\Big)^{1/2} \cdot \Big(E\frac{1}{|B'_\epsilon+B_s-x|^2}\Big)^{1/2} ds\\
 &=& C\epsilon^{1/2} \int_0^t\Big(E\frac{1}{|B_{s+\epsilon}-x|^2}\Big)^{1/2} ds\leq  C\epsilon^{1/2} \Big(\int_0^t  E\frac{1}{|B_{s+\epsilon}-x|^2}ds\Big)^{1/2} \cdot \Big(\int_0^t 1^2 ds\Big)^{1/2}\\
 &\leq & C\epsilon^{1/2} \cdot t^{1/2} \Big(\int_0^{t+1} ds \int p_s(y) \frac{1}{|y-x|^2} dy\Big)^{1/2} \to 0 \text{ as } \epsilon \to 0 \ \text{ by Corollary 1}.
\end{eqnarray*}
The $L^1$ convergence of $M_t(P_\epsilon g_x)$ to $M_t(g_x)$ follows.\\
\\
(iv)

For the convergence of the last term in (10), by Lemma 5 we get
\begin{eqnarray*}
& & E \Big|\int_0^t X_s\big(\frac{\Delta}{2} P_{\epsilon} g_x) ds-\frac{1}{2}\int_0^t ds \int \frac{1}{|y-x|^2} X_s(dy) \  \Big|\\
&=& \frac{1}{2} \ E \Big|\int_0^t ds \int X_s(dy) \int p_{\epsilon}(y-z) \frac{1}{|z-x|^2} dz-\int_0^t ds \int  X_s(dy) \frac{1}{|y-x|^2} \Big|\\
&\leq& \frac{1}{2} \ E \int_0^t ds \int X_s(dy) \  \Big|\int p_{\epsilon}(y-z) \frac{1}{|z-x|^2} dz- \frac{1}{|y-x|^2} \Big|\\
&=& \frac{1}{2} \int_0^t ds \int_{\R^3} \Big|\int p_\epsilon(y-z) \frac{1}{|z-x|^2}dz-\frac{1}{|y-x|^2}\Big| p_s(y) dy.
\end{eqnarray*}
$\mathit{Claim:}$ \[\Big|\int p_\epsilon(y-z) \frac{1}{|z-x|^2}dz-\frac{1}{|y-x|^2}\Big| \to 0 \text{ as } \epsilon\to 0\text{ for } y\neq x .\]
\begin{proof}
For $\xi=y-x \neq 0$, 
\begin{eqnarray*}
& & \Big|\int p_\epsilon(y-z) \frac{1}{|z-x|^2}dz-\frac{1}{|y-x|^2}\Big|\\
&=& \Big|\int p_\epsilon(z) \frac{1}{|z-(y-x)|^2}dz-\frac{1}{|y-x|^2}\Big|\\
&\leq& \int p_\epsilon(z) \Big|\frac{1}{|z-\xi|^2}-\frac{1}{|\xi|^2}\Big| dz\\
&=& E\bigg(\Big|\frac{1}{|B_\epsilon-\xi|^2}-\frac{1}{|\xi|^2}\Big|\bigg)\\
&=& E \bigg(\Big||B_\epsilon-\xi|-|\xi|\Big| \cdot \Big(\frac{|B_\epsilon-\xi|+|\xi|}{|B_\epsilon-\xi|^2 |\xi|^2}\Big)\bigg)\\
&\leq&  E \bigg(|B_\epsilon| \cdot \Big(\frac{|B_\epsilon-\xi|+|\xi|}{|B_\epsilon-\xi|^2 |\xi|^2}\Big)\bigg)\\
&=& E\bigg(|B_\epsilon|\cdot \frac{1}{|B_\epsilon-\xi|^2 |\xi|}\bigg)+E\bigg( |B_\epsilon| \cdot \frac{1}{|B_\epsilon-\xi|\ |\xi|^2}\bigg).
\end{eqnarray*}
For the first term, we use Holder's inequality with $1/p=1/5$ and $1/q=4/5$ to get 
\begin{eqnarray*}
& & E\Big(|B_\epsilon|\cdot \frac{1}{|B_\epsilon-\xi|^2 |\xi|}\Big)\leq \frac{1}{|\xi|} \cdot \Big(E(|B_\epsilon|^5)\Big)^{1/5} \cdot \Big(E\big((\frac{1}{|B_\epsilon-\xi|^2 })^{5/4}\big)\Big)^{4/5}.
\end{eqnarray*}
By Lemma 4, we have \[E\frac{1}{|B_\epsilon-\xi|^{5/2}} \leq C\cdot |\xi|^{-\frac{5}{2}}<\infty,\] 
so
\[E\Big(|B_\epsilon|\cdot \frac{1}{|B_\epsilon-\xi|^2 |\xi|}\Big) \leq \frac{1}{|\xi|} \Big(C |\xi|^{-\frac{5}{2}}\Big)^{4/5} \Big(E|B_\epsilon|^5\Big)^{1/5} \to 0 \text{ as } \epsilon \to 0.\]
Similarly \[E \Big(|B_\epsilon| \cdot \frac{1}{|B_\epsilon-\xi| |\xi|^2}\Big) \to 0.\]
\end{proof}
Note that we have just proved that \[|\int p_\epsilon(y-z) \frac{1}{|z-x|^2}dz-\frac{1}{|y-x|^2}| \to 0\]
almost everywhere ($y\neq x$) as $\epsilon \to 0$. Corollary 1 gives us \[\int_0^t ds \int_{\R^3} p_s(y) \frac{1}{|y-x|^2} dy<\infty,\] and by Lemma 4 \[\Big|\int p_\epsilon(y-z) \frac{1}{|z-x|^2}dz-\frac{1}{|y-x|^2}\Big| \leq (C+1) \frac{1}{|y-x|^2} \text{ for all } \epsilon,\]
by Dominated Convergence Theorem, \[\int_0^t ds \int_{\R^3} \Big|\int p_\epsilon(y-z) \frac{1}{|z-x|^2}dz-\frac{1}{|y-x|^2}\Big| p_s(y) dy \to 0,\] and we proved that \[\int_0^t X_s(\frac{\Delta}{2} P_{\epsilon} g_x)ds \xrightarrow[]{L^1}  \int_0^t \int \frac{1}{|y-x|^2} X_s(dy) ds.\]
\\
(v) Combining (i)-(iv), we build the $L^1$ convergence of each term in (10) to the corresponding term in (4), therefore (4) holds a.s. and the proof of Proposition 1 is done.$\hfill\square$

\subsubsection{Proof of Proposition 2}
Let $h_{\epsilon,x} (y)=\sqrt[]{|y-x|^2+\epsilon}$, then \[\nabla h_{\epsilon,x} (y)=\frac{y-x}{\sqrt[]{|y-x|^2+\epsilon}}\] and \[\Delta h_{\epsilon,x} (y)=\frac{(d-1)|y-x|^2+d \epsilon}{(|y-x|^2+\epsilon)^{3/2}}.\]
By Ito's Lemma, we have
\begin{eqnarray*}
\sqrt[]{|B_t-x|^2+\epsilon}=\sqrt[]{|x|^2+\epsilon}&+&\int_0^t \frac{B_s-x}{\sqrt[]{|B_s-x|^2+\epsilon}}\cdot dB_s\\
&+& \frac{1}{2} \int_0^t \frac{(d-1)|B_s-x|^2+d \epsilon}{(|B_s-x|^2+\epsilon)^{3/2}} ds
\end{eqnarray*}
Let $H_s=\frac{B_s-x}{\sqrt[]{|B_s-x|^2+\epsilon}}$, then \[M_t^\epsilon:=\int_0^t \frac{B_s-x}{\sqrt[]{|B_s-x|^2+\epsilon}}\cdot dB_s \in cM_{0,\text{loc}}.\] Since \[E[M^\epsilon]_t=E \int_0^t \frac{|B_s-x|^2}{|B_s-x|^2+\epsilon}ds \leq E \int_0^t 1 ds=t<\infty,\] then $M^\epsilon$ is a martingale and hence by taking expectation
\begin{equation}
E \sqrt[]{|B_t-x|^2+\epsilon}=\sqrt[]{|x|^2+\epsilon}+\frac{1}{2} \int_0^t E \frac{(d-1)|B_s-x|^2+d \epsilon}{(|B_s-x|^2+\epsilon)^{3/2}} ds.
\end{equation}
By Fatou's Lemma,
\begin{eqnarray*}
& & \frac{1}{2}\int_0^t E \frac{d-1}{|B_s-x|} ds= \frac{1}{2} \int_0^t  E \  \lim_{\epsilon \to 0} \frac{(d-1)|B_s-x|^2+d \epsilon}{(|B_s-x|^2+\epsilon)^{3/2}} ds\\
&\leq&  \frac{1}{2}\int_0^t  \ \liminf_{\epsilon \to 0} E\frac{(d-1)|B_s-x|^2+d \epsilon}{(|B_s-x|^2+\epsilon)^{3/2}} ds\\
&\leq&  \liminf_{\epsilon \to 0} \frac{1}{2} \int_0^t  \  E\frac{(d-1)|B_s-x|^2+d \epsilon}{(|B_s-x|^2+\epsilon)^{3/2}} ds\\
&=& \liminf_{\epsilon \to 0} \Big[E \sqrt[]{|B_t-x|^2+\epsilon}- \sqrt[]{|x|^2+\epsilon}\Big] \text{ by (11)}\\
&=& E |B_t-x|-|x|\leq E|B_t|.
\end{eqnarray*}
The last equality is from 
\[0\leq \sqrt[]{|x|^2+\epsilon} -|x| \leq \sqrt[]{\epsilon} \to 0,\]
and
\begin{eqnarray*}
0 &\leq& E \sqrt[]{|B_t-x|^2+\epsilon}-E|B_t-x|\\
&\leq& E\Big[|B_t-x|+\sqrt[]{\epsilon}\Big]-E|B_t-x|=\sqrt[]{\epsilon} \to 0.
\end{eqnarray*}
So 
\begin{eqnarray*}
&& \int_0^t \int \frac{1}{|y-x|} p_s(y) dy ds=\int_0^t E \frac{1}{|B_s-x|} ds \leq\frac{2}{d-1} E|B_t|<\infty.\\
\end{eqnarray*}
$\hfill\square$
\section{Proof of Theorem 2}
To prove Theorem 2, we need the Tanaka formula for $d=2$, which are stated below and the proof will follow after the proof of Theorem 2.

\begin{proposition}(Tanaka formula for d=2) 
Let $c_2=1/\pi$ and $g_x(y)= \log|y-x|$, where $x\neq 0$. Then we have a.s. that
\begin{equation}
L_t^x=c_2\Big[X_t(g_x)-\delta_0 (g_x)-M_t(g_x)\Big].
\end{equation}
\end{proposition}
\noindent $\mathbf{Remark.}$ Barlow, Evans and Perkins [1] gives a Tanaka formula for local time of Super-Brownian Motion in $d=2$, which is
\[X_t(g_{\alpha,x})=X_0(g_{\alpha,x})+M_t(g_{\alpha,x})+\alpha \int_0^t X_s(g_{\alpha,x}) ds-L_t^x,\] for all $t\geq 0$ a.s.. Here $g_{\alpha,x}(y)$  is defined to be $\int_0^\infty e^{-\alpha t} p_t(x-y) dt$. We can see that $g_{\alpha,x}$ is not well defined for $\alpha=0$ and our result effectively extends the Tanaka formula in [1] to the $\alpha=0$ case.

\subsection{Proof of Theorem 2}
By (12), note that $\delta_0(g_x)=\log|x|=-\log 1/|x|$, \[L_t^x-c_2\log \frac{1}{|x|}=c_2\Big[ X_t(g_x)-M_t(g_x) \Big],\] therefore \[E\Big|L_t^x-c_2\log \frac{1}{|x|}\Big| \leq c_2 E\Big|X_t(g_x)\Big|+  c_2 E\Big|M_t(g_x)\Big|.\]
For the first term, 
\begin{eqnarray*}
& & E\Big|X_t(g_x)\Big|\leq E X_t(|g_x|)= \int p_t(y) \Big|\log |y-x|\Big| dy\\
&\leq & \int p_t(y) \frac{1}{|y-x|} 1_{\{|y-x|<1\}} dy+\int p_t(y) |y-x| 1_{\{|y-x|\geq 1\}} dy\\
&\leq & \frac{1}{2\pi t} \int \frac{1}{|y-x|} 1_{\{|y-x|<1\}} dy+\int p_t(y) (|y|+|x|) dy\\
&=& \frac{1}{2\pi t} \cdot 2\pi+|x| +E|B_t| \to \frac{1}{2\pi t} \cdot 2\pi +E|B_t|.
\end{eqnarray*}
For the second term, \[\Big(E|M_t(g_x)|\Big)^2\leq E\Big(M_t^2(g_x)\Big)=E\int_0^t X_s(g_x^2) ds,\] 
and by Lemma 1
\begin{eqnarray*}
& & E\int_0^t X_s(g_x^2) ds=\int_0^t ds \int p_s(y) (\log|y-x|)^2 dy\\
&\leq& \int_0^t ds \int p_s(y) \frac{1}{|y-x|} 1_{\{|y-x|<1\}} dy+\int_0^t ds \int p_s(y) |y-x| 1_{\{|y-x|\geq 1\}} dy\\
&\leq& \int_0^t ds \int p_s(y) \frac{1}{|y-x|} dy+\int_0^t ds \int p_s(y) (|y|+|x|)  dy.
\end{eqnarray*}
By Proposition 2 in $d=2$, \[ \int_0^t ds \int p_s(y) \frac{1}{|y-x|} dy\leq 2 E|B_t|<\infty,\] and
\[\int_0^t ds \int p_s(y) (|y|+|x|)  dy=|x| t+\int_0^t E |B_s| ds \to \int_0^t E |B_s| ds <\infty.\]
Therefore \[\limsup_{x\to 0} E\Big|L_t^x-c_2\log \frac{1}{|x|}\Big|<\infty.\]

\subsection{Proof of Proposition 3}

\subsubsection{Some useful lemmas}
\begin{lemma}
In $\R^2$, for $0<\alpha<2$, there exists a constant $C=C(\alpha)$ such that for any $x\neq 0$ and $t>0$, 
\[
\int_{\R^2} p_t(y) \frac{1}{|y-x|^\alpha} dy<C \frac{1}{|x|^\alpha}.
\]
\end{lemma}

\begin{proof}
The proof follows from the proof of Lemma 4 after some modification. 
\end{proof}

\begin{corollary}
In $\R^2$, for any $0<\alpha<2$, there exists a constant $C=C(\alpha)$ such that for any $x\neq 0$ and $t>0$
\[
E\int_0^t \int \frac{1}{|y-x|^\alpha} X_s(dy) \ ds =\int_0^t ds \int_{\R^2} p_s(y) \frac{1}{|y-x|^\alpha} dy<C \frac{1}{|x|^\alpha} t.
\]
\end{corollary}

\begin{proof}
It follows from Lemma 6.
\end{proof}

\begin{lemma}
Let $g_x(y)= \log|y-x|$, where $x, y\in \R^2$, then for any $s>0$ and $y\neq x$, we have
\[\frac{\Delta_y}{2} P_s g_x(y)= \pi  p_s(y-x).\]
\end{lemma}
\begin{proof}
Idea of this proof is from Evans [3]. For any fixed $s>0$, $p_{s}(y)=(2\pi s)^{-1} e^{-|y|^2/{2s}} \in C_0^\infty(\R^2)$, we have \[\|Dp_s\|_{L^\infty(\R^2)}<\infty \text{ and } \|\Delta p_s\|_{L^\infty(\R^2)}<\infty.\]
For any $\delta \in (0,1)$, 
\begin{eqnarray*}
& & \Delta_y \int_{\R^2} p_{s}(y-z) g_x(z)dz\\
&=&  \int_{B(x,\delta)} \Delta_y p_{s}(y-z) g_x(z)dz+\int_{\R^2-B(x,\delta)} \Delta_y p_{s}(y-z) g_x(z)dz\\
&=:& I_\delta+J_\delta.
\end{eqnarray*}
Now \[|I_\delta|\leq \|\Delta p_s\|_{L^\infty(\R^2)} \int_{B(x,\delta)} |g_x(z)| dz\leq C \delta^2 |\log \delta| \to 0.\]
Note that $\Delta_y p_{s}(y-z)=\Delta_z p_{s}(y-z)$. Integration by parts yields
\begin{eqnarray*}
J_\delta &=& \int_{\R^2-B(x,\delta)} \Delta_z p_{s}(y-z) g_x(z)dz\\
&=& \int_{\partial B(x,\delta)} g_x(z) \frac{\partial p_{s}}{\partial \nu}(y-z) dz-\int_{\R^2-B(x,\delta)} D_z p_{s}(y-z) D_z g_x(z)dz\\
&=:& K_\delta+L_\delta,
\end{eqnarray*}
$\nu$ denoting the inward pointing unit normal along $\partial B(x,\delta).$ So
\[|K_\delta|\leq \|Dp_s\|_{L^\infty(\R^2)} \int_{\partial B(x,\delta)} |g_x(z)|dz \leq C \delta |\log \delta| \to 0.   \]
We continue by integrating by parts again in the term $L_\delta$ to find
\begin{eqnarray*}
L_\delta &=& \int_{\R^2-B(x,\delta)} p_{s}(y-z) \Delta_z g_x(z)dz-\int_{\partial B(x,\delta)} p_{s}(y-z) \frac{\partial g_x}{\partial \nu}(z) dz\\
&=& -\int_{\partial B(x,\delta)} p_{s}(y-z) \frac{\partial g_x}{\partial \nu}(z) dz\\
&=:& M_\delta.
\end{eqnarray*}
since $\Delta_z g_x(z)=0$ when $z$ is away from $x$.\\
\\
Now $Dg_x(z)=\frac{z-x}{|z-x|^2}(z\neq x)$ and $\nu=\frac{-(z-x)}{|z-x|}=\frac{-(z-x)}{\delta}$ on $\partial B(x,\delta)$. Hence $\frac{\partial g_x}{\partial \nu}(z)=\nu \cdot Dg_x(z)=-\frac{1}{\delta} $ on $\partial B(x,\delta)$. Since $2\pi \delta$ is the surface area of the sphere $\partial B(x,\delta)$ in $\R^2$, we have
\[M_\delta=2\pi \cdot \frac{1}{2\pi \delta} \int_{\partial B(x,\delta)} p_{s} (y-z)dz \to  2 \pi p_{s}(y-x) \text{ as } \delta \to 0.\]
Therefore we proved \[\frac{\Delta_y}{2} P_s g_x(y)= \pi p_s(y-x)= \pi p_s^x(y).\]

\end{proof}

\subsubsection{Proof of Proposition 3}
Using the same argument in proving Proposition 1 in Section 2.2.2, by a smooth cutoff $\chi_n$ of $\log$ and let $n$ goes to infinity, we have following equation hold a.s., 
 \begin{equation}
 X_t(P_{\epsilon} g_x)=\delta_0(P_{\epsilon} g_x)+ M_t(P_{\epsilon} g_x)+\int_0^t X_s\Big(\frac{\Delta}{2} P_{\epsilon} g_x\Big) ds,
 \end{equation}
 where $M_t(P_{\epsilon} g_x)$ is a martingale with quadratic variation being \[[M(P_{\epsilon} g_x)]_t=\int_0^t X_s((P_{\epsilon} g_x)^2) ds,\] and $M_t^2(P_{\epsilon} g_x)-[M(P_{\epsilon} g_x)]_t$ is also a martingale.\\
 
 Let $\epsilon$ goes to $0$, we will show the a.s. convergence of each term in (13) to the corresponding term in  Proposition 3, which is equivalent to 
\begin{equation}
X_t(g_x)=\delta_0(g_x)+M_t(g_x)+\pi L_t^x.
\end{equation}
By Lemma 7, we have \[\int_0^t X_s(\frac{\Delta}{2} P_\epsilon g_x) ds= \pi \int_0^t X_s(p_\epsilon^x) ds \xrightarrow[]{\text{a.s. }} \pi L_t^x \text{ as } \epsilon \to 0.\]
Then in (i)-(iii) we will build the $L^1$ convergence of the rest three terms in (13) to the corresponding term in (14) and we can take a subsequence along which all four terms converge a.s. and therefore (14) holds a.s.. 
\\
(i)

Let $B_t$ and $B_t^{'}$ be two independent standard Brownian motion in $\R^2$, 
\begin{eqnarray*}
& & E\Big|X_t(P_\epsilon g_x)-X_t(g_x)\Big| \leq E \Big[X_t\Big(|P_\epsilon g_x-g_x|\Big)\Big]\\
&=&  \int p_t(y) dy \Big|\int p_\epsilon(z) \log|z-(y-x)|dz-\log|y-x|\Big|\\
&\leq&  \int p_t(y) dy \int p_\epsilon(z) \Big|\log|z-(y-x)|-\log|y-x|\Big| dz\\
&=&  E\Big[\Big|\log \frac{|B'_\epsilon-(B_t-x)|}{|B_t-x|}\Big|\Big]\\
&\leq&   E \sqrt{\frac{|B'_\epsilon| }{|B_t-x|}}+ E \sqrt{\frac{|B'_\epsilon|}{|B'_\epsilon+B_t-x|}} \text{ by Lemma 1. }
\end{eqnarray*}
Since $E\sqrt[]{|B_\epsilon|} \to 0$ and  $E\sqrt[]{\frac{1}{|B_t-x|}}<\infty$ by Lemma 6,
\[E \sqrt{\frac{|B'_\epsilon| }{|B_t-x|}}= E \sqrt[]{|B'_\epsilon|} \cdot E\sqrt[]{\frac{1}{|B_t-x|}} \to 0.\]
For the second term, we use Cauchy Schwarz Inequality,
\[\bigg(E \sqrt{\frac{|B'_\epsilon|}{|B'_\epsilon+B_t-x|}}\bigg)^2 \leq E|B'_\epsilon| \cdot E\frac{1}{|B'_\epsilon+B_t-x|}= E|B'_\epsilon| \cdot E\frac{1}{|B_{t+\epsilon}-x|}.\]
So by Lemma 6 \[E \sqrt{\frac{|B'_\epsilon|}{|B'_\epsilon+B_t-x|}} \leq \bigg(E|B'_\epsilon|\bigg)^{1/2} \cdot \bigg(E\frac{1}{|B_{t+\epsilon}-x|}\bigg)^{1/2} \to 0 \text{ as } \epsilon \to 0.\]
and the $L^1$ convergence of $X_t(P_\epsilon g_x)$ to $X_t(g_x)$ follows.\\

(ii)
Similarly we have
\begin{eqnarray*}
& & \bigg|\delta_0(P_\epsilon g_x)-\delta_0(g_x)\bigg|=\bigg|\int_{\R^2}  p_\epsilon(y) \log|y-x| dy-\log|x| \bigg|\\
&\leq & \int_{\R^2}  p_\epsilon(y) \Big|\log|y-x| -\log|x|\Big| dy=E\Big|\log|B_\epsilon-x|-\log|x|\Big|\\
&\leq&  E \sqrt[]{\frac{|B_\epsilon|}{|x|}}+E\sqrt[]{\frac{|B_\epsilon|}{|B_\epsilon-x|}} \to 0.
\end{eqnarray*}
(iii) For the convergence of the martingale term $M_t(P_\epsilon g_x)$ to $M_t(g_x)$,
\begin{eqnarray*}
& & \Big(E|M_t(P_\epsilon g_x)-M_t(g_x)|\Big)^2 \leq E\Big[\Big(M_t(P_\epsilon g_x)-M_t(g_x)\Big)^2\Big]\\
&=& E \int_0^t X_s\Big((P_\epsilon g_x-g_x)^2\Big) ds\\
&=& \int_0^t ds \int p_s(y) dy \Big(\int p_\epsilon(z) (\log|z-(y-x)|-\log|y-x|) dz\Big)^2\\
&\leq&  \int_0^t ds \int p_s(y) dy \int p_\epsilon(z) \Big(\log|z-(y-x)|-\log|y-x|\Big)^2 dz\\
&=& \int_0^t  E \Bigg[\Bigg(\log \frac{|B'_\epsilon-(B_s-x)|}{|B_s-x|}\Bigg)^2\Bigg] ds\\
&\leq& \int_0^t  E \Bigg[\Bigg(\sqrt[]{\frac{|B'_\epsilon|}{|B_s-x|}}+\sqrt[]{\frac{|B'_\epsilon|}{|B'_\epsilon-(B_s-x)|}}\Bigg)^2\Bigg] ds \hfill{\text{ (by Lemma 1)}}\\
&\leq& 2 \int_0^t  E \Bigg({\frac{|B'_\epsilon|}{|B_s-x|}}+{\frac{|B'_\epsilon|}{|B'_\epsilon-(B_s-x)|}}\Bigg) ds\\
&:=& 2J.
\end{eqnarray*}
For the first term in $J$, by Corollary 2,
\begin{eqnarray*}
\int_0^t  E \frac{|B'_\epsilon| }{|B_s-x|} ds=E |B'_\epsilon| \cdot \int_0^t  E\frac{1}{|B_s-x|} ds\to 0.
\end{eqnarray*}
and for the second term in $J$, using Holder's inequality with $1/p=1/3$ and $1/q=2/3$ twice, we get
\begin{eqnarray*}
& & \int_0^t  E \frac{|B'_\epsilon| }{|B'_\epsilon+B_s-x|} ds\leq  \int_0^t  
\Big(E (|B'_\epsilon|^3)\Big)^{1/3}\cdot \Big(E\frac{1}{|B'_\epsilon+B_s-x|^{3/2}}\Big)^{2/3} ds \\
&=& \Big(E (|B'_\epsilon|^3)\Big)^{1/3} \cdot \int_0^t  \Big(E\frac{1}{|B_{s+\epsilon}-x|^{3/2}}\Big)^{2/3} ds \\
&\leq&  \Big(E (|B'_\epsilon|^3)\Big)^{1/3} \cdot \Big(\int_0^t  E\frac{1}{|B_{s+\epsilon}-x|^{3/2}} ds\Big)^{2/3} \cdot \Big(\int_0^t 1^3 ds\Big)^{1/3} \\
&\leq&  \Big(E (|B'_\epsilon|^3)\Big)^{1/3} \cdot \Big(\int_0^{t+1}  \int E \frac{1}{|B_{s}-x|^{3/2}} ds\Big)^{2/3} \cdot t^{1/3} \to 0
\end{eqnarray*}
as $\epsilon \to 0$ by Corollary 2.\\
This completes the proof of Proposition 3.$\hfill\square$

\end{document}